\documentclass[oneside, 10pt]{amsart}

\usepackage{amsmath, amsfonts, amssymb, amsthm}

\newtheorem{theorem}{Theorem}[section]
\newtheorem{lemma}[theorem]{Lemma}
\newtheorem{corollary}[theorem]{Corollary}

\theoremstyle{definition}

\newtheorem{remark}[theorem]{Remark}

\numberwithin{equation}{section}

\begin{document}

\title[On the  Forelli--Rudin projection  theorem]{On the Forelli--Rudin projection theorem}

\author{Marijan Markovi\'{c}}

\begin{abstract}
Motivated by the Forelli--Rudin projection theorem we give in this paper a criterion for
boundedness of an  integral operator  on weighted  Lebesgue       spaces in the interval
$(0,1)$. We also  calculate     the precise norm of  this integral operator. This is the
content  of the first  part of the paper. In the second part, as applications,   we give
some results concerning   the Bergman projection and  the Berezin transform. We derive a
generalization of  the Dostani\'{c} result  on the norm of the Berezin transform  acting
on  Lebesgue spaces  over  the unit ball  in          $\mathbf{C}^n$.
\end{abstract}



\address{
Faculty of Natural Sciences and Mathematics\endgraf
University of Montenegro\endgraf
Cetinjski put b.b.\endgraf
81000 Podgorica\endgraf
Montenegro}

\email{marijanmmarkovic@gmail.com}

\maketitle

\section{The main result and its proof}
\subsection{Gauss hypergeometric functions}
We  firstly  recall  some basic facts concerning the  Gauss hypergeometric functions. For
$a,\, b,\, c\in\mathbf{C}$    the Gauss hypergeometric  function is given   by the series
\begin{equation*}
{_2}{F}_1(a,b;c;z) = \sum_{k=0}^\infty\frac{(a)_k(b)_k}{(c)_k}\frac {z^k}{k!}
\end{equation*}
for  all $c$  different from zero and negative integers.  Here, for   any complex  number
$q$   the shifted factorial  (the Pochhammer symbol) is
\begin{equation*}
(q)_\beta= \left\{
\begin{array}{ll}
q(q + 1)\cdots (q + \beta  - 1),   & \hbox{if $\beta\ge1$,} \\
1,                                 & \hbox{if $\beta=0$},
\end{array}
\right.
\end{equation*}
where     $\beta$ is a non--negative  integer.
The above series converges at least    for $|z|<1$, and for
$|z|=1$ if $\Re ({c-a-b})>0$.

We will  need  the following three  identities:

i) For $\Re c>\Re d >0$ there  holds
\begin{equation}\label{EQ.INT.TRANS.1}
{_2}F_1 (a, b; c; z) = \frac {\Gamma(c)}{\Gamma(d) \Gamma(c-d)}
\int_0^1 t^{d-1} (1-t)^{c-d-1}\, {_2}F_1 (a, b; d; tz)\, dt,
\end{equation}
where  $z$  is   different from $1$  and $\left|\arg(1-z)\right|<\pi$. This is known  as
the Euler formula.

ii) The Euler transform says that
\begin{equation}\label{EQ.SIMPLE.TRANSFORM}
{_2}F_1 (a, b; c; z) =  (1-z)^{c-a-b}\,  {_2}F_1 (c-a, c- b; c, z).
\end{equation}

iii) Gauss proved that,  if $\Re (c-a-b)>0$, then
\begin{eqnarray}\label{EQ.GAUSS}
{_2}{F}_1(a,b;c;1) =
\frac{\Gamma(c) \Gamma(c-a-b)}{\Gamma(c-a)\Gamma(c-b)}.
\end{eqnarray}

For  all these facts we refer to  the second chapter in~\cite{AAR.BOOK.SPECIAL},   where
the reader may also find all properties concerning the Gauss hypergeometric functions we
need in the paper.

\subsection{The main result}
For          $\mu>0$ we denote   by  $L^p_\mu(0,1)$ the space  of all measurable function
$\varphi(t)$  in  $(0,1)$                                   which   satisfy the condition
\begin{equation*}
\|\varphi \|_{p,\mu}^p =    \mu  \int_0^1  |\varphi (t)|^p t^{\mu-1} dt<\infty.
\end{equation*}
The normalized        weighted   measure $\mu  t^{\mu - 1} dt$ we   denote by $d\mu (t)$.

For  a  parameter $\sigma>-1$ we  will consider the operator $F_\sigma$  given    in the
following way
\begin{equation*}
{F}_\sigma  \varphi (s)= \mu
\int_0^1 (1-t)^\sigma \, {_2}F_1(\lambda,\lambda;\mu;s\, t)\, \varphi (t)\, t^{\mu-1} dt,
\end{equation*}
where we have denoted
\begin{equation*}
\lambda =  (\mu +   \sigma   + 1)/2.
\end{equation*}
The operator   $F_\sigma$      may be viewed  as an integral operator  on $L^p_\mu(0,1)$
with the kernel
\begin{equation*}
K_\sigma (s,t) =(1-t)^\sigma\, {_2}F_1 (\lambda, \lambda;\mu;s\, t).
\end{equation*}
It    happens that the operator  $F_\sigma$ is bounded on $L^p_\mu (0,1)$ if and only if
$\sigma   > 1/p -1$.  This is the content of the following

\begin{theorem}
For            $1\le p < \infty$ the  operator ${F}_\sigma$ maps continuously  the space
$L^p_\mu(0,1)$   into itself  if and only if $\sigma> 1/p -1$. Moreover,         we have
\begin{equation*}
\|{F}_\sigma\|_{L^p_{\mu}(0,1)\rightarrow L^p_{\mu}(0,1)}  =
\frac {\Gamma(\mu+1)}{\Gamma^2(\lambda)}\Gamma(1/p) \Gamma(\sigma+1- 1/p)
\end{equation*}
for  all $\sigma> 1/p -1$.
\end{theorem}

If          $T:X\rightarrow Y$ is a linear operator  from  a  linear space with  a norm
$(X,\|\cdot\|_X)$  into $(Y,\|\cdot\|_Y)$, we denote by $\|T\|_{X\rightarrow Y}$    the
norm of $T$, i.e.,
\begin{equation*}
\|T\| _ { X\rightarrow Y} \ = \sup_{\|x\|_{X}\le 1 } {\|Tx\|_Y}.
\end{equation*}

\subsection{Auxiliary results}
The case  $p=1$  of our theorem    is not difficult to consider. It will be derived from

\begin{lemma}\label{LE.NORM.L1}
Let $\nu$ be a finite measure on $X$. Let $T$ be an integral operator which    acts   on
$L^1 = L^1(X,\nu)$  with the non--negative kernel $K(x,y)$,                  i.e.,   let
\begin{equation*}
T f (x) =\int_X   K(x,y)\, f(y)\, d\nu(y).
\end{equation*}
Then $T$ maps $L^1$ into itself if and only if
\begin{equation*}
\sup_{y\in X}  \int_X  K(x,y) \, d\nu(x)<\infty.
\end{equation*}
In this case we have
\begin{equation*}
\|T\|_{ L^1 \rightarrow L^1} =   \sup_{y\in X} \int_X  K(x,y) \, d\nu(x).
\end{equation*}
\end{lemma}

\begin{proof}
For all  $f\in L^1$ there holds
\begin{equation*}\begin{split}
\| {T} f\|_1&  =  \int_X  \left|\int_X  K(x,y) \, f(y)\, d\nu(y)\right|  d\nu(x)
\\&  \le \int_X \left\{ \int_X  K (x,y)\, d\nu(x)\right\} |f(y)| \, d\nu(y)
\\&  \le \left\{\sup_{y\in X} \int_X K(x,y)|\, d\nu(x)\right\} \|f\|_1.
\end{split}\end{equation*}

If we take $f\equiv 1$, then we have the equality  sign at each place above. This proves
the necessary condition, and gives the norm of $T$.
\end{proof}

In the case                  $1<p<\infty$    we will use the following well known result.

\begin{lemma}[The Schur test]
Suppose  that $(X,\nu)$ is a $\sigma$-finite measure space and $K(x,y)$ is a nonnegative
measurable      function on        $X\times X$, and $T$ the associated integral operator
\begin{equation*}
Tf(x) = \int_X K(x,y)\, f(y)\, d\nu(y).
\end{equation*}
Let  $1<p<\infty$ and $1/p + 1/q = 1$. If  there exist a positive constant $C$     and a
positive measurable function $f$  on $X$ such that
\begin{equation*}
\int_X K(x,y)\, f(y)^q\, d\nu(y)\le C\, f(x)^q.
\end{equation*}
for almost every $x\in X$ and
\begin{equation*}
\int_X K(x,y)\, f(x)^p\, d\nu(x)\le C\, f(y)^p.
\end{equation*}
for  almost every $y\in X$, then $T$ is bounded on $L^p  = L^p(X,\nu)$ and the following
estimate of the norm holds
\begin{equation*}
\|T\|_{ L^p\rightarrow L^p} \le C.
\end{equation*}
\end{lemma}

\subsection{The proof of the main result}
We start now with the proof of our main  result.

Let us first discus the simple case $p=1$.          By the Euler formula  and  the Euler
transform  we have
\begin{equation*}\begin{split}
\sup_{t\in (0,1)} \int_0^1 K_\sigma (s,t)\, d\mu (s) & \ =
\sup_{t\in (0,1)} \mu (1-t)^{\sigma} \int_0^1 s^{\mu-1}\, {_2}F_1 (\lambda, \lambda;\mu; s\, t)\, ds
\\&=\sup_{t\in (0,1)} \mu  (1-t)^{\sigma} \mu^{-1}\,   {_2}F_1(\lambda,\lambda;\mu+1;t)
\\&=\sup_{t\in (0,1)}   {_2}F_1 (\mu+1 - \lambda,\mu+1 -  \lambda ;\mu+1;t)
\\& = \frac {\Gamma (\mu+1 ) \Gamma (2\lambda -\mu-1 )}{\Gamma^2(\lambda)}
= \frac {\Gamma (\mu+1 )}{\Gamma^2 (\lambda )}\Gamma (\sigma )<\infty
\end{split}\end{equation*}
if and only if $\sigma>0$. The last conclusion  follows from

\begin{lemma}\label{LE.MAX}
For reals $y>0$ and $x$ the function  $\, {_2}F_1(x,x;y;r)\, $ is bounded in $(0,1)$
if and only  if  $y>2x$ in which case we have
\begin{equation*}
\sup_{r\in (0,1) }  {_2}F_1(x,x;y;r)\, =
 \, {_2}F_1(x,x;y;1)= \frac { \Gamma(y)\Gamma(y-2x) }{\Gamma^2(y-x)}.
\end{equation*}
\end{lemma}

\begin{proof}
If  $y>2x$, then  $\, {_2}F_1(x,x;y;r)\, $ is continuous in $[0,1]$ and increasing in
$r\in(0,1)$.    Therefore, we may apply the   Gauss relation to   obtain  the maximum.
In other cases we have
\begin{equation*}
{_2}F_1(x,x;y;r)  \sim  \frac {\Gamma\left(2x\right)}
{\Gamma^2\left(x\right)}\log \frac 1{1-r} \quad \text{if}\quad  y=2x,
\end{equation*}
and
\begin{equation*}
{_2}F_1(x,x;y;r) \sim \frac {\Gamma (y)\Gamma(2x - y)}
{\Gamma^2(x)} \frac 1{(1-r)^{2x-y}} \quad \text{if}\quad  y<2x.
\end{equation*}
For these asymptotic relations we refer to the second chapter in~\cite{AAR.BOOK.SPECIAL}.
\end{proof}

Now,         according to Lemma~\ref{LE.NORM.L1} we conclude
\begin{equation*}
\|{F_\sigma}\|_{ L^1_\mu (0,1)\rightarrow  L^1_\mu (0,1)}=
 \frac {\Gamma (\mu+1 )}{\Gamma^2 ( \lambda)}\Gamma (\sigma )
\end{equation*}
for all $\sigma>0$.

The rest of the proof of  our main result is devoted to the case $1<p<\infty$. It has  two
parts.  In  the first one,  using the Schur test, we   prove  that $F_\sigma$ is   bounded
for $\sigma> 1/p - 1$, and we obtain the estimate of the norm of ${F}_\sigma$ from   above.
In the    second part we  deliver the proof that the same number      is the norm estimate
from below.    In this part we also obtain that the condition $\sigma> 1/p-1$ is necessary
for the  boundedness of $F_\sigma$.

\subsubsection*{Part I}
To  start with this,  denote $\varphi(t)=(1-t)^{- 1/{p  q}}$. Assume that $\sigma> 1/p-1$.
We have
\begin{equation*}\begin{split}
&\int_0^1 K_\sigma(s,t)\, \varphi(t)^q\, d\mu(t) \\&=
\mu \int_0^1  t^{\mu-1}\, (1-t)^\sigma {_2}F_1(\lambda,\lambda;\mu;s\, t)\, \varphi(t)^q\, dt
\\& = \mu\, \frac {\Gamma(\mu) \Gamma(2\lambda-\mu- 1/p)}  {\Gamma(2\lambda-1/p)} \,  {_2}F_1(\lambda,\lambda;2\lambda- 1/p;s)
\\& = \frac {\Gamma(\mu+1) \Gamma(2\lambda-\mu- 1/p)}  {\Gamma(2\lambda- 1/p)}\,  (1-s)^{1/p}\,
{_2}F_1(\lambda,\lambda;2\lambda- 1/p;s)\, \varphi(s)^q
\\& =  \frac {\Gamma(\mu+1) \Gamma(2\lambda-\mu-1/p)}  {\Gamma(2\lambda-1/p)}\,
{_2}F_1(\lambda- 1/p, \lambda-  1 /p;2\lambda-  1/p;s)\, \varphi(s)^q
\\&\le \frac {\Gamma(\mu+1) \Gamma(2\lambda-\mu-1/p)} {\Gamma(2\lambda- 1/p)}
\,  \frac {\Gamma(2\lambda- 1/p) \Gamma( 1/p)}{\Gamma^2(\lambda)}\, \varphi(s)^{q}
\\&= \frac {\Gamma(\mu+1)} {\Gamma^2(\lambda)}\Gamma(\sigma+1-1/p)\Gamma(1/p)\, \varphi(s)^{q}.
\end{split}\end{equation*}
Since $(1-t)^\sigma \varphi(t)^q = (1-t)^{(2\lambda -1/p)-\mu-1}$, at the second place we
used       the Euler formula~\eqref{EQ.INT.TRANS.1} for   $c= 2\lambda- 1/ p$ and $d=\mu$
(observe that $c-d = \sigma+1/q> 1/ p- 1+ 1/ q =0$). At the third place we used       the
Euler  transform~\eqref{EQ.SIMPLE.TRANSFORM}.  The inequality follows                  by
Lemma~\ref{LE.MAX}, since the function ${_2}F_1(\lambda- 1/p,\lambda-1/p;2\lambda-1/p;s)$
is   increasing in $s\in (0,1)$.

Similarly, one derives
\begin{equation*}\begin{split}
&\int_0^1  K_\sigma(s,t)\, \varphi(s)^p\, d\mu(s)
\\&=\mu (1-t)^\sigma \int_0^1 s^{\mu-1}\,  {_2}F_1(\lambda,\lambda;\mu;s\, t)\, \varphi(s)^p \, ds
\\&=\mu (1-t)^{\sigma} \frac{\Gamma(\mu) \Gamma(1/p)} {\Gamma(\mu+1/p)}
\, {_2}F_1(\lambda,\lambda;\mu+ 1/p;t)
\\&=  \frac{\Gamma(\mu+1)  \Gamma(1/p)} {\Gamma(\mu+ 1/p)} (1-t)^{2\lambda-\mu- 1/p}
\, {_2}F_1(\lambda,\lambda;\mu+ 1/p;t)\, \varphi(t)^{p}
\\& = \frac{\Gamma(\mu+1)  \Gamma(1/p)} {\Gamma(\mu+ 1/p)}
\, {_2}F_1(\mu- \lambda+ 1/p,\mu -\lambda+  1/p;\mu+ 1/p;t)\, \varphi(t)^{p}
\\& =  \frac{\Gamma(\mu+1)  \Gamma(1/p)} {\Gamma(\mu+ 1/p)}\,
{_2}F_1(\mu  - \lambda+ 1/p,\mu - \lambda+ 1/p;\mu+ 1/p;t)\, \varphi(t)^{p}
\\&\le \frac{\Gamma(\mu+1) \Gamma(1/p)} {\Gamma(\mu+1/p)} \frac {\Gamma(\mu+ 1/p) \Gamma(2\lambda - \mu -1/p)}
{\Gamma^2(\lambda)}\, \varphi(t)^p
\\&=\frac {\Gamma(\mu+1)}{\Gamma^2(\lambda)}\Gamma(1/p) \Gamma(\sigma+1-1/p)\, \varphi(t)^p.
\end{split}\end{equation*}

By the Schur test we finally  obtain
\begin{equation*}
\|{F}_\sigma\|_ {L^p_\mu(0,1)\rightarrow L^p_\mu (0,1)}\le
\frac {\Gamma(\mu +1)}{\Gamma^2(\lambda)}\Gamma(1/p) \Gamma(\sigma+1- 1/p)
\end{equation*}
for every  $\sigma> 1/p -1$.

\subsubsection*{Part II}
As   we have said, the     aim of the second part is to establish the norm  estimate  of
$F_\sigma\,  (1<   p<\infty,\, \sigma>1/p-1)$         from  below. The  following   four
observations   will be useful in that approach.

i) If    $H(t)=C\, t^{ \theta/ p} (1-t)^{{\tilde \theta}/p}$, where    $C$ is a positive
constant,       then $H \in L^p_\mu(0,1),\,   1<p<\infty$     and $\|H \|_{p,\mu} =1$ if
and only if $\theta>-\mu,\, \tilde{\theta}>  -1$,                                    and
\begin{equation*}
C = \mu^{- 1/p} \mathrm{B}(\theta+\mu,\tilde{\theta}+1)^{-1/p}.
\end{equation*}
Indeed,  this follows from the simple  computation
\begin{equation*}\begin{split}
1  &  = \|H\|_{p,\mu}^p =\mu \int_0^1 H(t)^p  t^{\mu -1} dt
  = \mu  C^p  \int_0^1  t^{\theta+\mu-1}  (1-t)^{\tilde{\theta}}  dt
\\&=  \mu  C^p  \mathrm{B}(\theta+\mu ,\tilde{\theta}+1).
\end{split}\end{equation*}

ii) There holds  the identity
\begin{equation}\label{LE.INT.TRANS.2}
\int_0^1  t^{c-1}  (1-t)^{d-1}\, {_2}F_1 (a, b; c; t)\, dt =  \frac{\Gamma (c) \Gamma(d)
\Gamma(c+d-a-b)}{\Gamma(c+d-a) \Gamma(c+d-b)}
\end{equation}
for $\Re c>0,\, \Re d>0$ and $\Re(c+d-a-b)>0$.   We refer to~\cite{LIU.ZHOU.IEOT} for a
proof.

iii)
Let $l>0$ and let $G$ be any      function defined in an interval $(0,l)$  with positive
values.   For  every  $1<p<\infty$ we have
\begin{equation}\label{LE.LIMSUP}
\limsup_{(\zeta,\eta)\rightarrow (0,0)}
\frac{G(\eta)^{-1/p} G(( {\zeta+\eta})/p)} {G({\zeta}/({p-1}))^{1-1/p}}\ge 1.
\end{equation}
It is enough to note that if we set  $\eta=\zeta/({p-1})$,  then we have $(\zeta+\eta)/p
=\zeta/ (p-1)  =\eta$,  and therefore
\begin{equation*}
\frac{G(\eta)^{- 1/p} G(( {\zeta+\eta})/ p  )} {G({\zeta}/({p-1}))^{1- 1/p}}
= \frac{G(\eta)^{- 1/p} G(\eta)} {G(\eta)^{1-  1/p}}=  1,
\end{equation*}
what                                   immediately  implies the statement of this lemma.

iv) If        $L^p=L^p(X,\nu)$  is a Lebesgue space, recall that for an operator $T:L^p
\rightarrow L^p\, (1< p<\infty)$ we  have
\begin{equation*}\begin{split}
&\|T\|_{ L^p\rightarrow  L^p}
\\&= \sup \left\{ \left|\int _X T\Phi(x)\, \overline{\Psi(x)}\, d\nu(x)\right|: f\in\Phi
\in  L^p,\,   \Psi\in L^q,\,   \|\Phi\|_p=\|\Psi\|_q=1\right\},
\end{split}\end{equation*}
where $q$ is conjugate to $p$, i.e., $q= { p}/({p-1})$.

Therefore, for our    operator  $F_\sigma$ we will  calculate
\begin{equation*}
\int _0^1  F_\sigma \Phi(s)\, \overline{\Psi(s)}\,  d\mu (s)
\end{equation*}
for  appropriate  $\Phi (s)\in L^p_\mu(0,1)$ and $\Psi (t)\in L^q_\mu (0,1)$.  First of
all, using the Fubini theorem we obtain
\begin{equation*}\begin{split}
&\int _0^1 F_\sigma  \Phi(s)\,  \overline {\Psi(s)}\, d\mu(s)
\\& = \mu^2  \int_0^1 s^{\mu-1}   \left\{\int_0^1  t^{\mu-1}  (1-t)^{\sigma}\,
 {_2}F_1(\lambda,\lambda ;\mu;s\, t)\, \Phi (t)\, dt\right\} \overline{\Psi (s)}\, ds
\\&=\mu^2 \int_0^1  t^{\mu-1}  (1-t)^{\sigma} \left\{\int_0^1  s^{\mu-1}\, {_2}F_1(\lambda,\lambda;\mu;s\, t)\,
\overline{\Psi (s)}\, ds\right\} \Phi (t)\, dt.
\end{split}\end{equation*}
In              the  preceding relation we  will  take for  $\Phi(t)$ and   $\Psi(s)$
the  functions of the following  form
\begin{eqnarray*}
\Phi(t) = C \, t^{ \theta/ p}  (1-t)^{{\tilde{\theta}}/p},\quad
\Psi (s) = \tilde{C}\, s^{ \vartheta/ q}  (1-s)^{{\tilde{\vartheta}}/q}.
\end{eqnarray*}
We must have $\theta,\, \vartheta>-\mu$, and $\tilde{\theta},\, \tilde{\vartheta}>-1$,
as well as
\begin{equation*}
{C}^p  =\mu^{-1}  \mathrm{B}(\theta+\mu,\tilde{\theta}+1)^{-1},
\quad \tilde{C}^q = \mu^{-1}  \mathrm{B}(\vartheta+\mu,\tilde{\vartheta}+1)^{-1}.
\end{equation*}
In the sequel we will chose $\theta,\, \tilde{\theta}$ and $\vartheta,\, \tilde{\vartheta}$
in the way that it makes simpler the calculation of integrals in    the expression for
$\int _0^1 F_\sigma \Phi(s) \overline{\Psi(s)}\, d\mu(s)$.

Introducing           the preceding type of functions with $\vartheta=0$  we obtain
\begin{equation*}\begin{split}
&\int_0^1  s^{\mu-1}\,  {_2}F_1(\lambda,\lambda;\mu;s\, t)\, \overline{\Psi (s)}\, ds
\\&=\tilde{C}\, \int_0^1 s^{\mu-1}  (1-s)^{{\tilde{\vartheta}}/q}\,  {_2}F_1(\lambda,\lambda;\mu;s\, t)\, ds
\\&=\tilde{C}\, \int_0^1  s^{\mu-1}  (1-s)^{({\tilde{\vartheta}}/q+\mu+1)-\mu-1}
\, {_2}F_1(\lambda,\lambda;\mu;s\, t)\, ds
\\& = \tilde{C}\,  \frac{\Gamma(\mu)  \Gamma( {\tilde{\vartheta}}/q+1)}
{\Gamma({\tilde{\vartheta}}/q+\mu+1)}\, {_2}F_1(\lambda,\lambda;{\tilde{\vartheta}}/q+\mu+1;t),
\end{split}\end{equation*}
where   we have used~\eqref{EQ.INT.TRANS.1} for $ {c}={\tilde{\vartheta}}/q+\mu+1$
and ${d}=\mu$;         note ${c}-{d} = {\tilde{\vartheta}}/q+1>-{1}/ q+1 = 1/p> 0$.

For the sake of  simplicity in the following calculation  we will take  $\tilde{\vartheta} =
({\theta-p})/({p-1})$.  Then we have
\begin{equation*}
{\tilde{\vartheta}} /q+\mu +1   =  \mu + \theta/p.
\end{equation*}
Since we must have $\tilde{\vartheta}>-1$, it follows that $\theta>1$.

Now, it remains to transform
\begin{equation*}\begin{split}
&\int_0^1 t^{\mu-1} (1-t)^{\sigma}\,   {_2}F_1(\lambda,\lambda;{\tilde{\vartheta}}/q+\mu+1;t)\, \Phi(t)\, dt
\\&=C\, \int_0^1 t^{(\mu + \theta/ p)-1} (1-t)^{({\tilde{\theta}}/p+\sigma +1)-1}\,
{_2}F_1(\lambda,\lambda;\mu +\theta/p;t)\, dt
\\&= C \, \frac {\Gamma(\mu+  \theta/ p )  \Gamma({\tilde{\theta}}/ p+\sigma +1) \Gamma((\theta  + {\tilde{\theta}})/p)}
{\Gamma^2((\theta  +  {\tilde{\theta}})/p + \lambda)}.
\end{split}\end{equation*}
We have used~\eqref{LE.INT.TRANS.2}; note that $(\mu+\theta/p)  + ({\tilde{\theta}}/p+\sigma +1) - \lambda
= (\theta  +  {\tilde{\theta}})/p>0$.

All together  we have
\begin{equation*}\begin{split}
&\int_0^1 F_\sigma  \Phi(s)\,  \overline{\Psi(s)}\, d\mu (s)
\\&= \mu \, C \, \tilde{C}\,
\frac {\Gamma(\mu+1)  \Gamma(\mu+ {{\theta}}/p)  \Gamma( {\tilde{\theta}}/p+ \sigma +1)
\Gamma((\theta  + {\tilde{\theta}})/p)}{\Gamma^2((\theta  + {\tilde{\theta}})/p +\lambda)}.
\end{split}\end{equation*}

In           the sequel we assume that  $\sigma>1/p-1$ (note that  in this case we have
${\tilde{\theta}}/p+\sigma +1>0$). If  $\sigma\le 1/p-1$, then   $F_\sigma:L^p_\mu(0,1)
\rightarrow L^p_\mu(0,1)$ is not well defined.

Since $\vartheta = 0 $ and $\tilde{\vartheta} =({\theta-p})/({p-1})$  we obtain
\begin{equation*}\begin{split}
C\, \tilde{C}
&=  \mu ^{-  1/p}  {\mathrm B}(\theta+\mu,\tilde{\theta}+1)^{- 1/p}  \mu^{-  1/q}  {\mathrm B}(\vartheta+\mu,\tilde{\vartheta}+1)^{1/p-1}
\\&\sim \mu^{-1} \frac {\Gamma(\tilde{\theta}+1)^{-  1/p}} {\Gamma (({\theta-1})/ ({p-1}) )^{1- 1/p}},
\end{split}\end{equation*}
as ${(\theta,\tilde{\theta})\rightarrow (1,-1)}$.
Now, regarding~\eqref{LE.LIMSUP} it follows
\begin{equation*}\begin{split}
&\limsup_{(\theta,\tilde{\theta})\rightarrow (1,-1)}  \int_0^1  F_\sigma  \Phi(s)\, \overline{\Psi(s)}\, d\mu (s)
\\& = \frac {\Gamma (\mu+1)}{\Gamma^2(\lambda)}\Gamma( 1/p) \Gamma(-1/p+\sigma +1)
\limsup_{(\theta,\tilde{\theta})\rightarrow (1,-1)}
\frac {\Gamma(\tilde{\theta}+1)^{-1/p} \Gamma((\theta +{\tilde{\theta}})/p)}{\Gamma(({\theta-1})/({p-1}))^{1-1/p}}
\\&= \frac { \Gamma(\mu+1)}
{\Gamma^2 (\lambda )}\Gamma(1/p)  \Gamma(-1/p+\sigma +1)  \limsup_{(\zeta,\eta)\rightarrow (0,0)}
\frac {\Gamma(\eta)^{-1/p} \Gamma( (\zeta +\eta)/ p)}{\Gamma({\zeta}/({p-1}) )^{1-1/p}}
\\&\ge\frac { \Gamma(\mu+1)}{\Gamma^2 (\lambda)}\Gamma(1/p)\Gamma(-1/p+\sigma +1).
\end{split}\end{equation*}
Thus,  we have proved
\begin{equation*}\begin{split}
\|F_\sigma\|_{L^p_\mu(0,1)\rightarrow L^p_\mu(0,1)}&\ge
\limsup_{(\theta,\tilde{\theta})\rightarrow (1,-1)} \int _0^1  F_\sigma  \Phi(s)\,  \overline{\Psi(s)} \, d\mu (s)
\\&\ge\frac {\Gamma(\mu+1)}{\Gamma^2 ( \lambda)}\Gamma( 1/p)\Gamma(- 1/p+\sigma +1)
\end{split}\end{equation*}
for   $\sigma>1/p-1$.

\section{The Bergman projection and the Berezin transform}
\subsection{The Bergman projection}
In the sequel  $n$  will  be a positive integer.  Let
\begin{equation*}
\left<z,w\right>= z_1\overline{w}_1+\dots+z_n\overline{w}_n,
\end{equation*}
stand for the  inner product in the complex $n$-dimensional space  $\mathbf{C}^n$, where
$z=(z_1,\dots,z_n)$  and $w=(w_1,\dots,w_n)$.       The  standard norm in $\mathbf{C}^n$,
induced by the inner product, is denoted by $|z|=\sqrt{\left<z,z\right>}$.  We denote by
$B$ the unit ball   $\{z\in\mathbf C^n:|z|<1\}$ in $\mathbf{C}^n$.   Let $S= \partial B$
be the unit sphere. The normalized Lebesgue measure on the unit ball (sphere) is denoted
by $dv\, (d\tau)$.

Following  the Rudin  monograph~\cite{RUDIN.BOOK.BALL} as well as the Forelli and Rudin
work~\cite{FORELLI.RUDIN.INDIANA},  associate with each complex number  $s=\sigma+it,\,
\sigma>-1$ the integral kernel
\begin{equation*}
K_s(z,w)=\frac{\left(1-|w|^2\right)^s}{\left(1-\left<z,w\right>\right)^{n+1+s}},
\end{equation*}
and let
\begin{equation*}
T_s f(z) = c_s \int_B K_s(z,w)\, f(w)\, dv(w),\quad z\in B.
\end{equation*}
We understand $f(w)$ is a such one function in  $B$ that the previous integral  is well
defined,      and the complex power is understood to be the principal   branch.     The
operator $T_s$   is   the Bergman projection. The  coefficient $c_s$  is  chosen in the
way that for the weighted    measure in the unit ball
\begin{equation*}
dv_s(w)=c_s   (1-|w|^2 )^s dv(w)
\end{equation*}
there  holds $v_s(B)=1$, i.e., $T_s 1=1$.                   Using the polar coordinates
\begin{equation*}
\int_B h(z)\, dv(z) = 2n \int_0^1 r^{2n-1} dr \int_S h(r\zeta)\,  d\tau(\zeta),
\end{equation*}
one can show that
\begin{equation*}
c_s^{-1}     = n  \mathrm {B} (s+1,n) = \frac{\Gamma(s+1)
\Gamma(n+1)}{\Gamma(n+s+1)},
\end{equation*}
where $\Gamma$ and $\mathrm{B}$ are Euler functions.

Let $L^p(B)\, (1\le p<\infty)$ stand  for the Lebesgue space of all measurable functions
in the unit ball of $\mathbf{C}^n$ which modulus  with the  exponent   $p$ is integrable.
For $p=\infty$    let it be the space of all  essentially  bounded measurable  functions.
Denote       by $\|\cdot\|_p$ the usual  norm on  $L^p(B)\,  (1\le p\le \infty)$. Recall
that
\begin{equation*}
\|f\|_p^p \ = \int_{B}  |f(z)|^p\, dv(z)
\end{equation*}
for $f\in L^p(B)\, (1\le p<\infty)$.

Forelli    and    Rudin~\cite{FORELLI.RUDIN.INDIANA} proved that   $T_s:L^p(B) \rightarrow
L^p_a = L^p\cap H(B)$, where $H(B)$  is the space   of all analytic functions in  the unit
ball, is a bounded  (and surjective)   operator if and only if $\sigma >  1/p -1$,   where
$1\le p<\infty$. Moreover,   they find        $\|T_s\|_{L^1(B)\rightarrow L_a^1(B)}$   for
$\sigma>0$   and $\|T_s\|_{ L^2(B)\rightarrow L^2_a(B)}$ for $\sigma>- 1/2$. It seems that
the  calculation of $\|T_s\|_{ L^p(B)\rightarrow L^p_a(B)}$ in other cases is not  an easy
problem.

On  the other hand, if  $p=\infty$,  it is known~\cite{CHOE.PAMS}     that the   operator
$T_\sigma\, (\sigma>-1)$ projects  $L^\infty(B)$  continuously  onto the Bloch      space
$\mathcal{B}$  of  the   unit ball in  $\mathbf{C}^n$.    Recall that the  Bloch    space
$\mathcal{B}$   contains   all  functions $f$   analytic in $B$  for which the semi--norm
$\|f\|_{\beta}   =    \sup_{z\in B}    \left(1-|z|^2\right) \left|\nabla f(z)\right|$  is
finite.   One can obtain a true norm by adding $|f(0)|$,           more precisely  in the
following way
\begin{equation*}
\|f\|_\mathcal{B}=|f(0)|+\|f\|_\beta,\quad f\in\mathcal{B}.
\end{equation*}

The $\beta$-(semi-)norm of $T_\sigma:L^{\infty}(B)\rightarrow \mathcal{B}$ is defined by
\begin{equation*}
\|T_\sigma\|_\beta   \  = \sup_{\|f\|_\infty\le 1}\|T_\sigma f\|_\beta.
\end{equation*}
In~\cite{KALAJ.MARKOVIC.MATH.SCAND}     we find the (semi-)norm of $T_\sigma$      w.r.t.
the $\beta$-(semi-)norm. We obtained
\begin{equation*}
\|T_\sigma\|_\beta = \frac{\Gamma(2\lambda+1)}{\Gamma^2(\lambda +1/2)},
\end{equation*}
where  we have introduced
\begin{equation*}
\lambda =(n+\sigma+1)/2.
\end{equation*}
Following  the approach as in~\cite{PERALA.AASF.2013}, one     can derive
\begin{equation*}
\|T_\sigma\|_{{L^\infty(B)\rightarrow \mathcal{B}} }= 1+\|T_\sigma\|_\beta = 1+\frac{\Gamma(2\lambda+1)}
{\Gamma^2(\lambda+1 /2)}.
\end{equation*}
Particulary, for  $\sigma=0$ and $n=1$ we put $P = T_0$ and  $B=\mathbf{U}$ (then we have
the original Bergman projection). Per\"{a}l\"{a}   proved that
\begin{equation*}
\|P\|_\beta =\frac 8\pi\quad\text{and}\quad\|P\|_{L^\infty(\mathbf{U})\rightarrow\mathcal{B}}  =
1 + \frac 8\pi,
\end{equation*}
which are the main      results  from~\cite{PERALA.AASF.2012} and~\cite{PERALA.AASF.2013},
respectively.            For a   related result we refer to~\cite{KALAJ.VUJADINOVIC.JOT}.

\subsection{The maximal Bergman projection}
Beside  the operator                $T_\sigma$ we will consider now the integral operator
$\tilde{T}_\sigma\,  (\sigma>-1)$                                     given by the kernel
\begin{equation*}
|K_\sigma(z,w)| = \frac{\left(1-|w|^2\right)^\sigma}{\left|1-\left<z,w\right>\right|^{2\lambda}},
\end{equation*}
i.e.,
\begin{equation*}
\tilde{T}_\sigma f(z)=
c_\sigma \int_B \frac{\left(1-|w|^2\right)^\sigma}{\left|1-\left<z,w\right>\right|^{2\lambda}}\, f(w)\, dv(w),\quad z\in B.
\end{equation*}
It       is  known that $\tilde{T}_\sigma$  maps  $L^p(B)\, (1 \le p<\infty)$ into itself
continuously                                 if    and only  if     and $\sigma>  1/p-1$;
see~\cite{FORELLI.RUDIN.INDIANA} where  Forelli and Rudin used this operator  in order to
establish  the continuity of $T_\sigma$.

In order to connect  our main result with the Forelli--Rudin result, we need to transform
the integral $I_c(z)$ which appears in the first  chapter                    of the Rudin
monograph~\cite{RUDIN.BOOK.BALL}. For the proof of the next lemma  see Proposition 1.4.10
in~\cite{RUDIN.BOOK.BALL}.

\begin{lemma}\label{LE.I.C}
For  any real number $c$ introduce
\begin{equation*}
I_c (z) \, = \int_S \frac{d\tau(\zeta)}{\left|1-\left<z,\zeta\right>\right|^{n+c}},\quad
z\in B.
\end{equation*}
Then
\begin{equation*}
I_c(z)\, =\, {_2}F_1(\tilde{\lambda},\tilde{\lambda};n;|z|^2),
\end{equation*}
where $\tilde{\lambda} = ( n +  c)/2$.
\end{lemma}

The following  simple observations   will be useful.

Let $h(w) =H(|w|^2)$ be a radially symmetric function in the unit ball, where $H(t)$      is
defined in the interval $(0,1)$ and non--negative.

i) Norm of $h\in L^p (B)$ is given by
\begin{equation*}
\|h\|_p^p\ = n \int_0^1\, s^{n-1}\, H(s)^p\, ds = \|H\|_{p,n}^p
\end{equation*}
for  $1\le p<\infty$. Indeed,  using polar coordinates  we obtain
\begin{equation*}\begin{split}
\int_B h(z)\, dv(z)& = 2n  \int_0^1 r^{2n-1}\, dr  \int_S  h(r\zeta)\, d\tau(\zeta)
\\&= 2n \int_0^1\, r^{2n-1}\, H (r^2)\, dr
=n \int_0^1\, s^{n-1}\, H (s)\, ds.
\end{split}\end{equation*}

ii) The  function  $\tilde{T}_\sigma  h (z)$ is also radially symmetric,  if it is defined.
Moreover,
\begin{equation*}\begin{split}
\tilde{T}_\sigma h(z)&  = c_\sigma\, n \, \int_0^1  t^{n-1}   (1-t )^\sigma\,
{_2}F_1(\lambda,\lambda; n;t\,  |z|^2)\, H(t)\, dt
\\& = c_\sigma  F_\sigma H (|z|^2).
\end{split}\end{equation*}
To see that this relation holds, use  polar coordinates and Lemma~\ref{LE.I.C} to   obtain
\begin{equation*}\begin{split}
c_\sigma^{-1} \tilde {T}_\sigma h(z)&
=2n \int_0^1  r^{2n-1}\,  (1-r^2)^\sigma\, I_{2\lambda-n}(rz)\, H (r^2)\, dr
\\&= n  \int_0^1  s^{n-1}\,  (1-s)^\sigma\, {_2}F_1(\lambda,\lambda;n;s|z|^2)\, H (s)\, ds
\end{split}\end{equation*}
(we introduced $s=r^2$ to obtain the last integral).

Thus,  the operator   $\tilde{T}_\sigma : L^p(B)\rightarrow L^p(B)$ is bounded if and only
if $F_\sigma: L^p_n (0,1) \rightarrow L^p_n (0,1)$ is   bounded. Moreover,
\begin{equation*}
\|\tilde{T}_\sigma\|_{L^p(B)\rightarrow L^p(B)} = c_\sigma  \|F_\sigma\|_{ L^p_n (0,1) \rightarrow L^p_n (0,1)}
\end{equation*}
for all $1\le p<\infty$ and $\sigma> 1/p-1$.

Therefore, we have

\begin{theorem}\label{TH.MAIN}
The operator    $\tilde{T}_\sigma$ is bounded if and only if $\sigma> 1/p-1$. Norm  of
$\tilde {T}_\sigma: L^p(B)\rightarrow L^p(B)\, (1\le p<\infty)$ is
\begin{equation*}\begin{split}
\|\tilde {T}_\sigma\|_{ L^p(B)\rightarrow L^p(B)} &
=\frac{\Gamma(2\lambda)}{\Gamma^2(\lambda)\Gamma(\sigma+1)}\Gamma(1/p)\Gamma(\sigma+1-1/p)
\\&= \frac{\Gamma(n+\sigma+1)}
{{\Gamma^2({(n+\sigma+1)/2})}\Gamma(\sigma+1)}\Gamma(1/p)\Gamma(\sigma + 1- 1/p)
\end{split}\end{equation*}
for all $\sigma>1/p-1$.
\end{theorem}

\begin{remark}
For $n=1$ Theorem~\ref{TH.MAIN} reduces to the main result in~\cite{DOSTANIC.J.ANAL.MATH}.
See Theorem 1 there. See also~\cite{LIU.ZHOU.IJM}.
\end{remark}

The    conjugate operator $\tilde{T}_\sigma^*: L^p(B)\rightarrow L^p(B)\, (1<p\le\infty)$
of   $\tilde{T}_\sigma: L^q(B)\rightarrow L^q(B)\, (1\le q<\infty)$ is
\begin{equation*}
\tilde{T}_\sigma^*g(z)
= c_\sigma  \int_B \frac{\left(1-|z|^2\right)^\sigma}{\left|1-\left<z,w\right>\right|^{2\lambda}}\, g(w)\, dv(w),\quad z\in B.
\end{equation*}
Since
\begin{equation*}
\|\tilde {T}^*_\sigma\|_{ {L^p(B)\rightarrow L^p(B)} } = \|\tilde {T}_\sigma \|_{{L^q(B)\rightarrow L^q(B)}},
\end{equation*}
we immediately  deduce

\begin{corollary}\label{EQ.NORM.CONJUGATE.TTILDA}
Norm   of  $\tilde {T}_\sigma^*:L^p(B)\rightarrow L^p(B)\, (1<p\le \infty)$ is given by
\begin{equation*}\begin{split}
\|\tilde {T}^*_\sigma\|_{ L^p(B)\rightarrow L^p(B)} &
=\frac { \Gamma(n+\sigma+1)} { \Gamma^2((n +\sigma + 1)/2)  \Gamma(\sigma+1)}   \Gamma( 1/q) \Gamma(\sigma + 1- 1/q)
\\&= c_\sigma \frac{\Gamma(n+1)}{\Gamma^2( (n+\sigma +1)/2)}\Gamma(1/q)\Gamma(\sigma + 1- 1/q)
\end{split}\end{equation*}
for  $\sigma>  1/q-1$.
\end{corollary}

\subsection{An estimate of the norm of $T_\sigma$}
Forelli and Rudin~\cite{FORELLI.RUDIN.INDIANA} proved that
\begin{equation*}
\|T_\sigma\|_{ L^1(B)\rightarrow L^1_a(B)}  = \frac{\Gamma(\sigma)}{\Gamma(\sigma+1)}\frac{\Gamma(2\lambda)}{\Gamma^2(\lambda)},
\quad \sigma>0.
\end{equation*}
Note that $\|\tilde{T}_\sigma\|_{ L^1(B)\rightarrow L^1(B)} = \|T_\sigma\|_{L^1(B)\rightarrow L^1_a (B)}$.

They also proved
\begin{equation*}
\|T_\sigma\|_{ L^2(B)\rightarrow L^2_a(B)}  = \frac{\sqrt{\Gamma(2\sigma + 1)} }{\Gamma(\sigma+1)}, \quad \sigma>-1/2.
\end{equation*}

Regarding theses results by the Riesz--Thorin theorem  we obtain
\begin{equation*}\begin{split}
\|T_\sigma\|_{L^p(B)\rightarrow L^p_a (B)}&
\le\left\{\frac{\Gamma(2\lambda)}{\Gamma^2(\lambda)}\frac{\Gamma(\sigma)}{\Gamma(\sigma+1)}\right\}^{2/p-1}
\left\{\frac{ \sqrt{\Gamma(2\sigma + 1)} }{\Gamma(\sigma+1)}\right\}^{2-2/p}
\\&= \frac {2^{1-1/p}} {\sigma^{1/p}}   \left\{\frac{  \Gamma(2 \sigma)}{\Gamma^2(\sigma)}\right\}^{1-1/p}
\left\{\frac{\Gamma(2\lambda)}{\Gamma^2(\lambda)}  \right\}^{2/p-1}
\end{split}\end{equation*}
for $1\le p\le 2$ and $\sigma>0$.

The estimate of $\|T_\sigma\|_{ L^p(B)\rightarrow L^p_a(B)}$  given  in the following
corollary,      which  follows from $\|T_\sigma\|_{ L^p(B)\rightarrow L^p_a (B)}  \le
\|\tilde{T}_\sigma\|_{ L^p(B)\rightarrow L^p(B)}$,            is better in some cases.

\begin{corollary}\label{CORO.NORM.ESTIMATE}
\begin{equation*}
\|T_\sigma\|_{ L^p(B)\rightarrow L^p_a (B)}  \le
\frac{\Gamma(2\lambda)} {\Gamma^2(\lambda)
\Gamma(\sigma+1)}\Gamma(1/p)\Gamma(\sigma + 1-1/p)
\end{equation*}
for $1\le p<\infty$ and $\sigma>\ 1/p -1$.
\end{corollary}

Particularly,        for $\sigma = 0$  we have $P=T_0$ and the norm estimate
\begin{equation*}
\|P\|_{  L^p(B)\rightarrow L^p_a(B)}\le
\frac{\Gamma(n+1)} {\Gamma^2((n+1)/2)} \frac{\pi}{\sin(\pi/p)},
\end{equation*}
where $1<p<\infty$. This estimate for $n=1$ is also obtained in~\cite{DOSTANIC.CZMJ}.

\subsection{$L^p$-norm of the  transform of Berezin}
In  the                    case of the unit ball  the Berezin transform takes the form
\begin{equation*}
\mathfrak{B}f(z)
=\int_B \frac { (1-|z|^2)^{n+1}} {\left|1-\left<z,w\right>\right|^{2n+2}}\, f(w)\, dv(w),\quad z\in B.
\end{equation*}
Berezin~\cite{BEREZIN.IZVESTIYA} introduced the notion  of covariant and contravariant
symbols of an operator. The Berezin transform finds applications       in the study of
Hankel and Toeplitz operators. An interesting result~\cite{ENGLIS.JFA.1994}  says that
if $u\in L^1(\mathbf{U})$, where $\mathbf{U}= \left\{z\in\mathbf{C}: |z|<1\right\}$ is
the unit disc in the complex plane, then $u$ is a harmonic function in $\mathbf{U}$ if
and only if $\mathfrak{B} u = u$.

Observe that
\begin{equation}\label{EQ.BEREZIN.TTILDA}
 {B} = c^{-1}_{n+1} \tilde{T}^*_{n+1}.
\end{equation}

\begin{corollary}
Norm  of the   Berezin  transform $\mathfrak{B}:L^p(B)\rightarrow L^p(B)$    is given by
\begin{equation*}
\|\mathfrak {B}\|_{ L^p(B)\rightarrow L^p(B)}
=\left\{ \prod_{k=1}^n\left( 1+\frac 1 {kp}\right)\right\}\frac {\pi/p} {\sin(\pi/p)}
\end{equation*}
for $1<p<\infty$, and
\begin{equation*}
\|\mathfrak{B}\| _{ L^\infty(B)\rightarrow L^\infty(B)} =1.
\end{equation*}
\end{corollary}

The result of this corollary is obtain in~\cite{LIU.ZHOU.IJM}, but we  give a proof for
the sake of completeness.

\begin{proof}
Let $1<p<\infty$.                 Since by Corollary~\ref{EQ.NORM.CONJUGATE.TTILDA} we have
\begin{equation*}
\|\tilde {T}^*_{n+1}\| _{L^p(B)\rightarrow L^p(B)} = c_{n+1} \frac{\Gamma(1- 1/p) \Gamma (n+1 +  1/p)}{\Gamma (n+1)},
\end{equation*}
it   follows by~\eqref{EQ.BEREZIN.TTILDA} that
\begin{equation*}\begin{split}
\|\mathfrak{B}\|_{L^p(B)\rightarrow L^p(B)} &= \frac{\Gamma(1-  1/p)\Gamma(n+1 + 1/p)}{\Gamma(n+1)}
\\&=\frac{\Gamma(1- 1/p)\left\{ \prod_{k=1}^n  (k + 1/p)\right\}     \Gamma(1/p)}{{n!}\, p}
\\&=\left\{ \prod_{k=1}^n\left( 1+\frac 1 {kp}\right)\right\}\frac {\pi/p}{\sin(\pi/p)}.
\end{split}\end{equation*}
We have used  the Euler identity $\Gamma(x) \Gamma(1-x) =  \frac {\pi}{\sin({\pi}/x)}$  for
$x\in (0,1)$ in order to obtain the last expression.

The case $p=\infty$ follows also from Corollary~\ref{EQ.NORM.CONJUGATE.TTILDA}. Introducing
$q=1$ we obtain
\begin{equation*}
\|\tilde {T}^*_{n+1} \|_{ L^\infty(B)\rightarrow L^\infty(B) } = c_{n+1},
\end{equation*}
what         implies the result concerning the $L^{\infty}$-norm of the Berezin   transform.
\end{proof}

\begin{corollary}
\begin{equation*}
\|{B}\|_{ L^2(B)\rightarrow L^2(B)}  =  \frac{(2n+ 1)!!}{( 2n)!!} {\frac\pi 2}.
\end{equation*}
\end{corollary}

\begin{corollary}
\begin{equation*}
\|{B}\|_{ L^p(B)\rightarrow L^p(B)}\sim  \frac {(n+1) \pi }{\sin(\pi/p)}\sim\frac {(n+1) \pi }{\pi -  \pi/ p}
\sim\frac {n+1 }{p - 1}, \quad
 p\rightarrow 1.
\end{equation*}
\end{corollary}


\begin{thebibliography}{10}

\bibitem{AAR.BOOK.SPECIAL}
G. Andrews, R. Askey and R. Roy, \emph{Special Functions}, Cambridge University Press, Cambridge, 1999.

\bibitem{BEREZIN.IZVESTIYA}
F. Berezin, \emph{Covariant and contravariant symbols of operators}, Math. USSR--Izv. \textbf{6} (1972), 1117--1151.

\bibitem{CHOE.PAMS}
B. Choe, \emph{Projections, the weighted Bergman spaces, and the Bloch space}, Proc. Amer. Math. Soc. \textbf{108} (1990), 127--136.

\bibitem{DOSTANIC.J.ANAL.MATH}
M. Dostani{\'c}, \emph{Norm of Berezin transform on $L^p$ space}, J. Anal. Math. \textbf{104} (2008), 13--23.

\bibitem{DOSTANIC.CZMJ}
M. Dostani\'{c}, \emph{Two sided norm estimate of the Bergman projection on $L^p$ spaces}, Czech Math. J. \textbf{58} (2008), 569--575.

\bibitem{ENGLIS.JFA.1994}
M. Engli\v{s}, \emph{Functions invariant under the Berezin transform}, J. Funct. Anal. \textbf{121} (1994), 233--254.

\bibitem{FORELLI.RUDIN.INDIANA}
F. Forelli and W. Rudin, \emph{Projections on spaces of holomorphic functions in balls}, Indiana Univ. Math. J. \textbf{24} (1974), 593-–602.

\bibitem{KALAJ.MARKOVIC.MATH.SCAND}
D. Kalaj and M. Markovi\'{c}, \emph{Norm of the Bergman projection}, Math. Scand. \textbf{115} (2014), 143--160.

\bibitem{KALAJ.VUJADINOVIC.JOT}
D. Kalaj and \DJ{}. Vujadinovi\'{c}, \emph{Norm of the Bergman projection onto the Bloch space},
J. Operator Theory, to appear.

\bibitem{LIU.ZHOU.IEOT}
C. Liu and L. Zhou, \emph{Norm of an integral operator related to the harmonic Bergman projection},
Integr. Equ. Oper. Theory \textbf{69} (2011), 557--566.

\bibitem{LIU.ZHOU.IJM}
C. Liu and L. Zhou, \emph{On the $p-$norm of the Berezin transform} Illinois J. Math. \textbf{56} (2012), 281--659.

\bibitem{PERALA.AASF.2012}
A. Per\"{a}l\"{a}, \emph{On the optimal constant for the Bergman projection onto the Bloch space},
Ann. Acad. Sci. Fenn. \textbf{37} (2012), 245--249.

\bibitem{PERALA.AASF.2013}
A. Per\"{a}l\"{a}, \emph{Bloch space and the norm of the Bergman projection}, Ann. Acad. Sci. Fenn. \textbf{38} (2013), 849--853.

\bibitem{RUDIN.BOOK.BALL}
W. Rudin, \emph{Function Theory in the Unit Ball of $\mathbf {C}^n$}, Springer--Verlag, New York, 1980.

\bibitem{ZHU.CONTEMP}
K. Zhu, \emph{A sharp norm estimate of the Bergman projection on $L^p$ spaces}, Contemp. Math. \textbf{404} (2006), 195--205.

\bibitem{ZHU.BOOK}
K. Zhu, \emph{Spaces of Holomorphic Functions in the Unit Ball}, Springer--Verlag, New York, 2005.

\end{thebibliography}
\end{document}